\documentclass[a4paper,10pt,reqno]{amsart}

\title{Lattice Point Counting in Sectors of Hyperbolic 3-space}

\author{Niko Laaksonen}
\thanks{The author was supported by the 150th Anniversary Postdoctoral Mobility Grant from the London Mathematical Society.}
\address{Department of Mathematics, University College London, Gower Street, London WC1E 6BT, United Kingdom}
\address{Department of Mathematical Sciences, University of Copenhagen, Universitetspark 5, 2100 Copenhagen \O, Denmark}
\email{n.laaksonen@ucl.ac.uk}

\date{\today}

\usepackage[backend=bibtex,bibencoding=ascii,style=numeric,firstinits=true]{biblatex}
\DeclareFieldFormat[article,inbook,incollection,inproceedings,patent,thesis,unpublished]{citetitle}{#1\isdot}
\DeclareFieldFormat[article,inbook,incollection,inproceedings,patent,thesis,unpublished]{title}{#1\isdot}

\renewcommand*{\intitlepunct}{\space}
\renewbibmacro{in:}{%
    \ifentrytype{article}{}{\printtext{\bibstring{in}\intitlepunct}}}
\DeclareFieldFormat[article]{volume}{\textbf{#1}}
\DeclareFieldFormat[article]{number}{\mkbibparens{#1}}
\renewbibmacro*{volume+number+eid}{%
    \printfield{volume}%
    \printfield{number}%
    \setunit{\addcomma\space}%
    \printfield{eid}}

\AtEveryBibitem{
    \clearfield{issn}
    \clearfield{url}
    \clearfield{bdsk-url-1}
    \clearfield{doi}
    \clearfield{isbn}
}
\DeclareNameAlias{sortname}{last-first}
\DeclareNameAlias{default}{last-first}

\usepackage{setspace} 
\usepackage[pdfauthor={Niko Laaksonen},
    pdftitle={Lattice Point Counting in Sectors of Hyperbolic 3-space},
    pdfkeywords={},
    pdfcreator={Pdflatex}]
{hyperref}
\hypersetup{colorlinks=true}

\usepackage{amssymb,amsthm}
\usepackage{mathtools} 

\usepackage[margin=1in]{geometry}

\usepackage{mathrsfs}

\usepackage{tikz}
\usetikzlibrary{calc}

\usepackage{dsfont} 

\usepackage{etoolbox} 

\usepackage{enumerate}

\newcommand{\reals}{\mathbb{R}}

\newcommand{\complex}{\mathbb{C}}
\newcommand{\uph}{\mathbb{H}^{2}}

\newcommand{\uphs}{\mathbb{H}^{3}}

\newcommand{\hypmodg}{\Gamma\backslash\hyps}
\newcommand{\tendsto}{\rightarrow}

\newcommand{\func}[3]{#1\!:\!#2\longrightarrow #3}

\newcommand{\hyps}{\mathcal{H}}
\newcommand{\hypp}{\mathcal{P}}

\newcommand{\slr}{\mathrm{SL}_{2}(\mathbb{R})}
\newcommand{\pslr}{\mathrm{PSL}_{2}(\mathbb{R})}
\newcommand{\pslz}{\mathrm{PSL}_{2}(\mathbb{Z})}

\newcommand{\pslc}{\mathrm{PSL}_{2}(\mathbb{C})}

\newcommand{\pslzi}{\mathrm{PSL}_{2}(\mathbb{Z}[i])}
\newcommand{\fixed}{H}
\newcommand{\cosets}{\fixed\backslash\Gamma}

\newcommand{\fund}[2][]{\mathcal{F}_{#1}\ifstrempty{#2}{}{(#2)}}
\newcommand{\id}{\mathrm{id}}
\newcommand{\period}{\hat{u}}

\newcommand{\ind}{\mathds{1}}

\newcommand{\proj}[1]{{#1}_{0}}

\renewcommand{\vec}[1]{\mathbf{#1}}
\renewcommand{\d}[1]{\,d#1}

\usepackage{xparse} 
\usepackage{mathtools} 

\DeclareDocumentCommand{\pd}{O{} O{} m}{\frac{\partial^{#1}#2}{\partial{#3}^{#1}}}
\DeclareDocumentCommand{\error}{O{X} O{p}}{E(#2,#1)}
\DeclareDocumentCommand{\eerror}{O{f} O{X} O{p}}{E_{#1}(#3,#2)}
\DeclareDocumentCommand{\ssum}{O{T} O{X} O{p}}{S_{#1}(#3,#2)}

\let\abs\undefined
\DeclarePairedDelimiter{\abs}{\lvert}{\rvert}
\DeclarePairedDelimiter{\norm}{\lVert}{\rVert}
\DeclarePairedDelimiterX{\inprod}[2]{\langle}{\rangle}{#1,#2}

\providecommand\given{}
\newcommand{\SetSymbol}{\nonscript\: :\nonscript\:\mathopen{}\allowbreak}
\DeclarePairedDelimiterX\Set[1]\{\}{%
    \renewcommand\given{\SetSymbol}
    #1
}

\DeclareMathOperator{\arcsinh}{arcsinh}
\DeclareMathOperator{\arccosh}{arccosh}

\DeclareMathOperator{\im}{Im}
\DeclareMathOperator{\stab}{Stab}
\DeclareMathOperator{\vol}{vol}

\newtheorem{thm}{Theorem}[section]
\newtheorem{lemma}{Lemma}[section]
\newtheorem{prop}{Proposition}[section]
\theoremstyle{definition}
\newtheorem{remark}{Remark}

\usepackage{abstract}

\begin{document}
	\maketitle
    \begin{center}
        \begin{minipage}{.7\textwidth}
            \begin{abstract}
                \footnotesize
                Let $\Gamma$ be a cocompact discrete subgroup of $\pslc$ and denote by $\hyps$ the three dimensional upper half-space.
                For a $p\in\hyps$,
                we count the number of points in the orbit $\Gamma p$, according to their distance, $\arccosh X$, from
                a totally geodesic hyperplane. The main term in $n$ dimensions was obtained by Herrmann
                for any subset of a totally geodesic submanifold. We prove a pointwise error term of $O(X^{3/2})$
                by extending the method of Huber and Chatzakos--Petridis to three dimensions.
                By applying Chamizo's large sieve inequalities we obtain the conjectured
                error term $O(X^{1+\epsilon})$ on average in the spatial aspect.
                We prove a corresponding large sieve
                inequality for the radial average and explain why it only improves on the pointwise bound by $1/6$.
            \end{abstract}
        \end{minipage}
    \end{center}

    \section{Introduction}\label{sec:intro}
    Let $\Gamma$ be a discrete group acting discontinuously on a hyperbolic space $\hyps$ and denote
    the quotient space by $M=\hypmodg$.
    The standard hyperbolic lattice point problem asks to count the number of points in the orbit $\Gamma p$ within a given distance from some fixed point $q\in\hyps$.
    For example, in two dimensions the counting function is
    \[N(z,w,X) = \#\Set{\gamma\in\Gamma\given 4u(\gamma z, w)+2\leq X},\]
    where $u$ is the standard point-pair invariant on $\uph$ and $z,w\in\uph$, and it measures
    the number of lattice points $\gamma z$ in a hyperbolic disc of radius $\arccosh(X/2)$
    centered at $w$. This problem was first considered by e.g.~Huber and Selberg. Selberg proved that for fixed $z,w\in\uph$,
    \[N(z,w,X)=\sqrt{\pi}\sum_{s_{j}\in(1/2,1]}\frac{\Gamma(s_{j}-1/2)}{\Gamma(s_{j}+1)}u_{j}(z)\overline{u}_{j}(w)+E(z,w,X),\]
    where the error term satisfies $E(z,w,X)=O(X^{2/3})$. The bound on the error term has
    not been improved for any cofinite $\Gamma$ or any choice of points $z,w\in\uph$.
    To find more evidence of the conjectured error term $E(z,w,X)=O(X^{1/2+\epsilon})$, it is useful to consider various averages.
    For example, \textcite{hill2005} look at the variance of the counting function in terms of the centre over the whole fundamental domain
    of any cofinite $\Gamma$ in hyperbolic $n$-space.
    For the case $\hyps=\uph$ and $\Gamma=\pslz$, with no eigenvalues $\lambda_{j}\leq 1/4$, their result is
    \[\int_{\Gamma\backslash\uph}\abs[\bigg]{N(z,w,X)-\frac{\pi}{\vol(\Gamma\backslash\uph)}X}^{2}\d{\mu(w)}=O(X),\]
    where $\mu(w)$ is the standard hyperbolic measure on $\uph$.
    On the other hand, \textcite{petridis2015} looked at a local average of $N(z,z,X)$ over $z$. Suppose that $f$ is smooth, non-negative, compactly supported
    function on $M$. For $\Gamma=\pslz$ they proved that
    \[\int_{\Gamma\backslash\uph}f(z)N(z,z,X)\d{\mu(z)}=\frac{\pi X}{\vol(\Gamma\backslash\uph)}\int_{\Gamma\backslash\uph}f(z)\d{\mu(z)}+O(X^{7/12+\epsilon}),\]
    where the error term depends on $\epsilon$ and $f$ only.
    This improves Selberg's bound halfway to the expected $1/2+\epsilon$. Their method requires
    knowledge of the average rate of QUE for Maa\ss~cusp forms on $M$ and other arithmetic information only available to groups similar to $\pslz$.
    In 1996~\textcite{chamizo21996} showed that it is possible to apply large sieve methods on $M$. As an application, he proved that by averaging
    over a large number of radii, one gets the expected bound on the error term $E(z,w,X)$:
    \begin{equation}\label{eq:chamizoresult}
        \frac{1}{X}\int_{X}^{2X}\abs{E(z,w,x)}^{2}\d{x}=O(X\log^{2}X).
    \end{equation}
    Furthermore, Chamizo also proves a similar result for the second and fourth moments of discrete averages over sufficiently spaced centres, which leads to
    \begin{equation}\label{eq:chamizoresult2}
        \biggl(\int_{\Gamma\backslash\uph}\abs{E(z,w,X)}^{2m}\d{\mu(z)}\biggr)^{\frac{1}{2m}}\!\!=O(X^{1/2}\log X),
    \end{equation}
    for $m=1,2$.

    Instead of measuring the distance between two points of $\hyps$, it is also possible to consider geodesic segments between various subspaces of $M$.
    In two dimensions,~\textcite{huber1956} looked at geodesic segments between a point and a fixed closed geodesic $\ell$.
    The geodesic $\ell$ corresponds to a hyperbolic conjugacy
    class $\mathfrak{H}$, given by some power $\nu$ of a primitive hyperbolic element $g\in\Gamma$. For cocompact $\Gamma$,
    Huber explained that counting
    \begin{equation}\label{eq:huberdef}
        N_{z}(T)=\#\Set{\gamma\in\mathfrak{H}\given d(z,\gamma z)\leq T}
    \end{equation}
    is equivalent to counting
    the geodesic segments from $z$ to $\ell$ according to length. If $\Gamma$ has no small eigenvalues,
    then Huber's main result in~\cite{huber1998} says that
    \begin{equation}\label{eq:huber}
        N_{z}(T)=\frac{2}{\vol(\Gamma\backslash\uph)}\frac{\mu}{\nu}X+O(X^{3/4}),
    \end{equation}
    where $\mu$ is the length of the invariant geodesic corresponding to $\mathfrak{H}$, and $X=\frac{\sinh T/2}{\sinh \mu/2}$.
    Independently,~\textcite{good1983} proved
    a stronger error bound of $O(X^{2/3})$. Good's methods also extend to more general counting problems for cofinite groups $\Gamma$.
    There is another interesting geometric interpretation that Huber gave for the counting problem in conjugacy classes.
    After conjugation we may assume that the geodesic $\ell$
    lies on the imaginary axis. Then the counting in $N_{z}(T)$ is equivalent to counting $\gamma z$ in the cosets $\gamma\in\Gamma/\langle g\rangle$, such that
    $\gamma z$ lies inside the sector formed by the imaginary axis and some angle $\Theta$.
    Chatzakos and Petridis~\cite{chatzakos2015} showed that it is possible to apply Chamizo's methods to obtain results analogous to~\eqref{eq:chamizoresult}
    and~\eqref{eq:chamizoresult2} for both cocompact and cofinite $\Gamma$. This was done by extending the method of Huber. Along the way they also
    obtain a new proof of Good's error term $O(X^{2/3})$.

    In $n$ dimensions, \textcite{herrmann1962} investigated the number of geodesic segments from a point to any Jordan measurable subset $Y$ of a totally
    geodesic submanifold $\mathcal{Y}\subset \hyps$.
    Let $N(r,Y,\Gamma p)$ be the number of orthogonal geodesic segments from $\gamma p$, for any $\gamma\in\Gamma$, to $Y$ with length at most $r$.
    For cocompact $\Gamma$, Herrmann proves that
    \begin{equation}\label{eq:herrmann}
        N(r, Y, \Gamma p)\sim \frac{2}{n-1}\frac{\pi^{(n-k)/2}}{\Gamma(\frac{n-k}{2})}\frac{\vol(Y)}{\vol(\Gamma\backslash\hyps)}\cosh^{n-1}r.
    \end{equation}
    His method is geometric, not depending on the action of the group $\Gamma$ on $\hyps$. He introduces an associated Dirichlet series and
    studies its analytic continuation. It is difficult to prove strong error terms with this method.
    We are interested in the error term of~\eqref{eq:herrmann} for $n=3$ and $k=n-1$. We study this for $Y=\mathcal{Y}$ by adapting the method
    of Huber and Chatzakos--Petridis
    to $\hyps=\uphs$ for cocompact $\Gamma\subset\pslc$.

    For the rest of this paper we fix $\hyps=\uphs$.
    We also focus solely on cocompact $\Gamma$, see Remark~\ref{rem:cocompact} for a discussion on cofinite groups.
    Let $\{u_{j}\}_{j\geq 0}$ be a complete orthonormal system of eigenfunctions of the Laplacian $\Delta$
    on $M$ with eigenvalues $\lambda_{j}=s_{j}(2-s_{j})\geq 0$.
    Let $\hypp$ be a totally geodesic hyperplane in $\hyps$ and define $v(p)=\arctan(x_{2}(p)/y(p))$.
    We prove the following theorem.
    \begin{thm}\label{thm1}
        Let $\Gamma$ be a cocompact discrete subgroup of $\pslc$. Set $H=\Gamma\cap\stab_{\pslc}(\hypp)$ and let
        $\period_{j}$ be the period integral of $u_{j}$ over the fundamental domain of $H$ restricted to $\hypp$.
        Define
        \begin{equation*}
            N(p,X)=\#\Set{\gamma\in\cosets\given(\cos v(\gamma p))^{-1}\leq X}.
        \end{equation*}
        Then
        \[N(p,X) = M(p,X) + \error,\]
        where
        \begin{equation}\label{eq:mainterm}
            M(p,X)= \frac{\vol(H\backslash\hypp)}{\vol(\hypmodg)}X^{2}+\sum_{1<s_{j}<2}\frac{2^{s_{j}-1}}{s_{j}}\period_{j}u_{j}(p)X^{s_{j}},
        \end{equation}
        \[\error = O(X^{3/2}).\]
    \end{thm}
    Here we understand $\vol(H\backslash\hypp)$ as the hyperbolic area in two dimensions.
    We also apply Chamizo's large sieve results in this case.
    For the radial average, Chamizo only provides a large sieve inequality in two dimensions.
    We generalise it to three dimensions and prove an improvement of $1/6$ on the pointwise
    error term on average.
    In~\cite[Appendix~B]{laaksonen2015} we show that the radial large sieve yields the same improvement for
    the error term in the standard hyperbolic lattice point problem in three dimensions.
    Due to structural reasons the improvements from the large sieve get worse for higher dimensions.
    We expect that in $n$ dimensions it is possible to obtain the conjectured bound for the second moment in the
    spatial aspect while for the radial aspect we get diminishing returns even for the mean square.
    We summarise our results in the following theorems. The corresponding discrete
    averages and more precise statements are given in Theorems~\ref{thm:two}~and~\ref{thm:three}.
    \begin{thm}
        Let $\Gamma$ be a cocompact discrete subgroup of $\pslc$. Then, for $X>2$,
        \[\frac{1}{X}\int_{X}^{2X}\abs{\error[x][p]}^{2}\d{x}\ll X^{2+2/3}\log X.\]
    \end{thm}
    \begin{thm}
        Let $\Gamma$ be a cocompact discrete subgroup of $\pslc$. Then, for $X>2$,
        \[\int_{\hypmodg}\abs{\error[X][p]}^{2}\d{\mu(p)}\ll X^{2}\log^{2}X.\]
    \end{thm}
    \begin{remark}
        It is also possible to obtain the radial mean square~\eqref{eq:chamizoresult} in the standard two dimensional lattice point problem
        by direct integration in the spectral expansion of the error term.
        This is done for a smoothed error term by~\textcite{phillips1994}. It is possible to deduce
        the result of Chamizo from their computations~\cite{cherubini2016}.
        It would be interesting to see if this can be done in our problem and whether it improves on the above estimate coming from the large sieve.
    \end{remark}
    \begin{remark}
        The other possible case in three dimensions, $k=1$, is substantially harder with our method. Currently, the spectral
        expansion of the automorphic function corresponding to $N(p,X)$ can be written in terms of $\period_{j}$ and
        an explicit solution to an ordinary differential equation. For $k=1$ we can no longer solve the corresponding
        eigenvalue equation as it remains a partial differential equation.
        Geometrically, $k=1$ corresponds to counting in a cone, while $k=2$ is counting in a sector.
    \end{remark}
    \begin{remark}
        The majority of computations in this paper are more explicit than in two dimensions (cf.~\cite{chatzakos2015}).
        This is because the Selberg transform, the spherical eigenfunctions and the special functions in the spectral expansion
        of the counting function~\eqref{eq:xiw} can all be expressed in an elementary form.
        We expect that there is always such a distinction between even and odd dimensions.
    \end{remark}
    \begin{remark}
        Dynamical systems and ergodic methods have also been applied to study lattice point counting. Their advantage
        is that the results generally apply to a larger set of manifolds. On the other hand, these methods fail to produce
        finer results, such as strong error terms.
        For example, \textcite{parkkonen2015}
        extend the counting in conjugacy classes problem of Huber to higher dimensions for loxodromic, parabolic and elliptic
        conjugacy classes for any discrete group of isometries $\Gamma$. See also~\cite{parkkonen2013} for a survey on a wider
        variety of counting problems analogous to \textcite{herrmann1962}.
        Moreover, \textcite{eskin1993} obtain main terms for a variety of counting problems on affine symmetric spaces
        defined by Lie groups. In particular, they give an alternate proof for the main term on homogenous affine varieties, which was also
        proved by \textcite{duke1993} through spectral methods.
    \end{remark}

    \section{The Geometric Setup of the Problem}
    We refer to~\cite{elstrodt1998} for the basic definitions concerning $\hyps$.
    We denote points $p\in\hyps$ by $p=z+yj=(x_{1},x_{2},y)$, where $z=x_{1}+ix_{2}\in\complex$ and $y>0$.
    The action of $\gamma=\begin{psmallmatrix}a & b\\ c&d\end{psmallmatrix}\in\Gamma$ on $\hyps$ is given by
    \[\gamma p=\left(\frac{(az+b)\overline{(cz+d)}+a\overline{c}y^{2}}{\norm{cp+d}^{2}},\frac{y}{\norm{cp+d}^{2}}\right),\]
    where $\norm{p}^{2}=\abs{z}^{2}+y^{2}$ is the Euclidean norm of $p$. Equipped with the hyperbolic metric, the hyperbolic distance between
    $p,\,p'\in\hyps$ is given by
    \[\cosh d(p,p')=\delta(p,p'),\]
    where $\delta$ is the standard point-pair invariant.
    This gives rise to the volume element
    \begin{equation}\label{eq:volume}
        d\mu(p)=\frac{dx_{1}\,dx_{2}\,dy}{y^{3}},
    \end{equation}
    and the hyperbolic Laplacian
    \[\Delta=y^{2}\left(\pd[2]{x_{1}}+\pd[2]{x_{2}}+\pd[2]{y}\right)-y\pd{y}.\]
    The totally geodesic hypersurfaces in $\hyps$ are the Euclidean hyperplanes and semispheres orthogonal to the complex plane.
    Motivated by Huber, we define a new set of coordinates. Let
    \begin{align*}
        x &= x_{1}, &
        u &= \log \sqrt{x_{2}^{2}+y^{2}}, &
        v &= \arctan\frac{x_{2}}{y},
    \end{align*}
    and transform to $p=(x(p), u(p), v(p))$. We often write $(x, u, v)$ for the same point as a shorthand if
    the point in question is clear. The effect of this change of coordinates on the metric and the Laplacian
    is summarised in the next lemma.
    \begin{lemma}\label{lem:coords}
        With the $(x,u,v)$ coordinates defined as above, we have
        \begin{align*}
            ds^{2} &= \frac{dx^{2}}{e^{2u}\cos^{2}v}+\frac{du^{2}+dv^{2}}{\cos^{2}v},\\
            d\mu(p) &= \frac{dx\,du\,dv}{e^{u}\cos^{3}v},
        \end{align*}
        and
        \[\Delta = e^{2u}\cos^{2}v\pd[2]{x} + \cos^{2}v\left(\pd[2]{u}+\pd[2]{v}\right)-\cos^{2}v\pd{u}+\sin v\cos v\pd{v}.\]
    \end{lemma}
    \begin{proof}
        The Jacobian of the transformation $(x_{1},x_{2},y)\mapsto(x,u,v)$ is
        \begin{equation*}
            \begin{pmatrix} 1 &  & \\
                 & \sin v\,e^{u} & \cos v\,e^{u}\\
                 & \cos v\,e^{u} & -\sin v\,e^{u}
            \end{pmatrix},
        \end{equation*}
        so that the hyperbolic metric tensor in these coordinates is
        \begin{equation}\label{eq:tensor}
            (g_{ij})=\frac{1}{\cos^{2}v}\begin{psmallmatrix} e^{-2u}\\
                & 1\\
                & & 1
            \end{psmallmatrix}.
        \end{equation}
        For the Laplacian we get
        \begin{align*}
            \Delta &= e^{u}\cos^{3}v\left(\pd{x}\left(\frac{e^{u}}{\cos v}\pd{x}\right)+\pd{u}\left(\frac{1}{e^{u}\cos v}\pd{u}\right)+\pd{v}\left(\frac{1}{e^{u}\cos v}\pd{v}\right)\right),
        \end{align*}
        which simplifies to the required form.
    \end{proof}
    \begin{figure}
        \centering
        \begin{tikzpicture}
            \draw [->, thick] (0, 0) -- (0, 4) node (taxis) [above] {$y$};
            \draw [->, thick] (30:-2) -- (30:3) node (xaxis) [above] {$x_{1}$};
            \draw [->, thick] (-2, 0) -- (3, 0) node (yaxis) [right] {$x_{2}$};

            \draw (30:-1.5) coordinate (p1) -- ++(0, 3) coordinate (p2) -- ++(30:3) coordinate (p3) -- ++(0, -3) coordinate (p4);
            \fill [opacity=0.4, gray] (p1) -- (p2) -- (p3) -- (p4) -- cycle;

            \draw [dashed, thin] (30:-1.5) -- ++(20:2) -- ++(30:3) -- ++(20:-2);

            \draw [->, thick] (0,0) -- (20:1.75) node (uaxis) [right] {$e^{u}$};
            \draw (0,0) ++(90:0.5) arc (90:20:0.5) node[midway, yshift=0.7em,xshift=0.2em] {$v$};
        \end{tikzpicture}
        \caption{\small The $(x,u,v)$ coordinates in $\hyps$.}
        \label{fig:coords}
    \end{figure}

    Now, let $\hypp$ be a totally geodesic hyperplane in $\hyps$. After conjugation by an element of $\pslc$,
    we may assume that $\hypp$ is given by the set $\Set{p\in\hyps\given v=0}$ (i.e.~$x_{2}=0$).
    Notice that in this case $v(p)$ measures the angle between $p$ and $\hypp$ as shown in Figure~\ref{fig:coords}.
    Let $p\in\hyps$. We denote the orthogonal projection (along geodesics) of $p$ onto $\hypp$ by $\proj{p}=(x(p),u(p),0)$.
    Next we identify all the elements of $\pslc$ that stabilise the plane $\hypp$. Since we are no longer working
    with a single geodesic, the stabiliser will be larger than in the two dimensional setting.
    \begin{lemma}
        The stabiliser of $\hypp\subset\pslc$ is
        \[\stab_{\pslc}(\hypp)=\pslr\bigcup\left(\begin{smallmatrix}i &\\ &-i\end{smallmatrix}\right)\pslr.\]
    \end{lemma}
    \begin{proof}
        Denote the stabiliser by $A$.
        Let $\gamma=\left(\begin{smallmatrix}a & b\\ c& d\end{smallmatrix}\right)\in A$. By definition
        we have that $\gamma p\in\hypp$ for any $p\in\hypp$, that is
        $x_{2}(p)=0$ implies that $x_{2}(\gamma p)=0$. Hence,
        \begin{equation*}
            x_{2}(\gamma p) = \frac{\im(a\overline{c}x_{1}^{2}+(a\overline{d}+b\overline{c})x_{1}+b\overline{d}+a\overline{c}y^{2})}{\norm{cp+d}^{2}}=0.
        \end{equation*}
        This needs to be true for any $x_{1}$ and $y$. Comparing coefficients we get
        \begin{align}
            \im(a\overline{c}) & = 0,\label{eq:trig1}\\
            \im(a\overline{d}+b\overline{c}) &=0,\label{eq:trig2}\\
            \im(b\overline{d}) &=0.\label{eq:trig3}
        \end{align}
        Now, write $a,b,c,d$ in polar form as
        \[\begin{aligned}
                a&=r_{a}e^{\theta_{a}i}, &   b&=r_{b}e^{\theta_{b}i},\\
                c&=r_{c}e^{\theta_{c}i}, &   d&=r_{d}e^{\theta_{d}i}.
            \end{aligned}\]
        Solving the equations~\eqref{eq:trig1},~\eqref{eq:trig2} and~\eqref{eq:trig3} tells us that
        either $\theta_{a}$, $\theta_{b}$, $\theta_{c}$, $\theta_{d}\in\{0,\pi\}$ or
        $\theta_{a}$, $\theta_{b}$, $\theta_{c}$, $\theta_{d}\in\{\frac{-\pi}{2},\frac{\pi}{2}\}$. In the first case the matrix is real and hence
        gives $\pslr$. In the second case we have, after considering the determinant,
        \[\gamma=\begin{pmatrix}i & \\ & -i\end{pmatrix}\begin{pmatrix}r_{a} & r_{b}\\ r_{c} & r_{d}\end{pmatrix}\quad\text{or}\quad\begin{pmatrix}r_{a} & r_{b}\\ r_{c} & r_{d}\end{pmatrix}\begin{pmatrix} i & \\ & -i\end{pmatrix}.\]
    \end{proof}
    Now define $H=\stab_{\pslc}(\hypp)\cap\Gamma$. We can then write the counting function as
    \[\widetilde{N}(p,\Theta)=\#\Set{\gamma\in\cosets\given\abs{v(\gamma p)}\leq \Theta},\]
    where $\Theta\in(0,\pi/2)$.
    If we set $(\cos \Theta)^{-1}=X$, then $\widetilde{N}$ takes on the following form (cf.~\textcite{huber1998})
    \[\widetilde{N}(p,\Theta)=N(p,X)=\#\Set*{\gamma\in\cosets\given\frac{1}{\cos v(\gamma p)}\leq X}.\]
    \begin{remark}
        Counting in the sector is equivalent to counting orthogonal geodesic segments from $\gamma p$ to $\hypp$ according to length.
        Hence, we can easily relate the main term~\eqref{eq:mainterm} to that of Herrmann's~\eqref{eq:herrmann}. Given a point $p\in\hyps$,
        the projection $p_{0}$ in the $(x_{1},x_{2},y)$-coordinates is given by $p_{0}=(x_{1},0,\sqrt{x_{2}^{2}+y^{2}})$. Then,
        by the explicit formula for the point-pair invariant, we get
        \[\delta(p,p_{0})=\frac{\sqrt{x_{2}^{2}+y^{2}}}{y}=\sec v.\]
        For $N( p, X)$ we have $X=\sec\Theta$ so that the maximal distance we are counting
        is $\arccosh\sec\Theta$. Substituting this into~\eqref{eq:herrmann} with $n=3$ and $k=2$ shows that the main terms agree.
    \end{remark}

    \begin{figure}[b]
        \centering
        \label{fig:fund}




        \begin{tikzpicture}
            \draw [->] (0, 0) -- (0, 4) node (taxis) [above] {$y$};
            \draw [->] (30:-2) -- (30:3.5) node (xaxis) [above] {$x_{1}$};
            \draw [->] (-4, 0) -- (4, 0) node (yaxis) [right] {$x_{2}$};

            \draw (30:-1.5) coordinate (p1) -- ++(0, 3) coordinate (p2) -- ++(30:3) coordinate (p3) -- ++(0, -3) coordinate (p4);

            \draw (90:0.7) coordinate (l2) +(0,2.3) coordinate (l1);
            \draw (30:0.5) ++(0,0.30) coordinate (l3)  +(0,2.7) coordinate (l4);
            \draw [fill opacity=0.5, thick, fill=black] (l1) -- (l2) .. controls ($ (l2) + (30:0.2) $) and ($ (l3) + (120:0.2) $) .. (l3) -- (l4);

            \draw [thick] (0,0) ++(1.0,0) coordinate (m2) -- +(1.7,0) coordinate (m1);
            \draw [thick] (l3) ++(-90:0.3) ++(0.3,0) coordinate (m3) -- +(2.3,0) coordinate (m4);
            \draw [thick] (m2) .. controls ($(m2) +(120:0.2)$) and ($(m3) +(-30:0.2)$) .. (m3);

            \draw [thick] (0,0) ++(-0.7,0) coordinate (n2) -- +(-1.9,0) coordinate (n1);
            \draw [thick] (l3) ++(-90:0.3) ++(-0.3,0) coordinate (n3) -- +(-2.4,0) coordinate (n4);
            \draw [thick] (n2) .. controls ($(n2) +(30:0.2)$) and ($(n3) +(210:0.2)$) .. (n3);

            \draw [densely dashed] (n2) to[in=100,out=90,distance=26] (m2);
            \draw [densely dashed] (n3) to[in=100,out=90,distance=12] (m3);
            \draw [dashed] (n1) to[in=90,out=90,distance=113] (m1);
            \draw [dashed] (n4) to[in=90,out=90,distance=113] (m4);
            \draw [densely dashed] (n1) to (n4);
            \draw [densely dashed] (m1) to (m4);
        \end{tikzpicture}
        \caption{\small Fundamental domain of $H$ in $\hyps$ with the part for $H\backslash\hypp$ highlighted (for $\Gamma=\pslzi$).}
    \end{figure}
    As we are working with invariance under $H$, we have to compute its fundamental domain,
    $\fund[\hyps]{H}$. It has a particularly convenient description in the $(x,u,v)$-coordinates.
    First, consider $H$ restricted to the plane $\hypp$ and denote the fundamental domain of $H$ in this space
    by $S=\fund[\uph]{H}$.
    \begin{lemma}
        We claim that $\fund[\hyps]{H}=F$, where $F$ is given by the union of rotated copies of $S$:
        \[F=\bigcup_{\theta\in(-\frac{\pi}{2},\frac{\pi}{2})}S_{\theta},\]
        and $S_{\theta}$ are defined by
        \[S_{\theta}=\Set{p\in\hyps\given v(p)=\theta, \proj{p}\in S}.\]
    \end{lemma}
    \begin{proof}
        This follows immediately from computing the action of $H$ on $\hyps$, which is seen to be independent of $v(p)$ in the $x$ and $u$-coordinates.
        First, let $\iota=\begin{psmallmatrix}i & \\ & -i\end{psmallmatrix}$,
        which acts on $p\in\hyps$ as a rotation by $\pi$ about the imaginary axis, $\iota p=\iota(z+yj)=-z+yj$,
        so that
        \begin{align*}
            x(\iota p) &= -x(p),
            & u(\iota p) & = u(p),
            & v(\iota p) & = -v(p).
        \end{align*}
        On the other hand, for $\tau=\left(\begin{smallmatrix}a & b\\ c&d\end{smallmatrix}\right)\in\pslr$, we find that
        \begin{align*}
            x(\tau p)
            &= \frac{ac\norm{p}^{2}+(ad+bc)x+bd}{c^{2}\norm{p}^{2}+2cdx+d^{2}}.
        \end{align*}
        And since $\norm{p}^{2}=x^{2}+e^{2u}$ it follows that $x(\tau p)$ does not depend on $v(p)$.
        Similarly,
        \begin{align*}
            x_{2}(\tau p)& = \frac{x_{2}}{\norm{cp+d}^{2}},
            & y(\tau p)&= \frac{y}{\norm{cp+d}^{2}},\\
            \shortintertext{so that}
            u(\tau p) &= \log\frac{\sqrt{x_{2}^{2}+y^{2}}}{\norm{cp+d}^{2}}  & v(\tau p) &= v(p),\\
            &= u(p)-\log\norm{cp+d}^{2}.\\
        \end{align*}
        It is easy to see that $\iota\tau\iota\in\slr$. Hence any $\gamma\in H$ can be written as $\iota\tau$, $\tau\iota$ or $\tau$.
        The first consequence of the above calculations is that for $p\in\hyps$ and $\gamma\in H$ the action of the group and the
        orthogonal projection to $\hypp$ commute, i.e.~$\proj{(\gamma p)}=\gamma\proj{p}$.
        Moreover, if
        $v(p)=\theta$ then $v(\gamma p)=\pm\theta$ for any $\gamma\in H$. Thus, suppose that
        $p,\gamma p\in S_{\pm\theta}$ for some $\theta\in(0,\pi/2)$ and $\gamma\in H$. It follows that
        $\proj{p}, \gamma\proj{p}\in S$, which is a contradiction as $S$ is a fundamental domain for $H$ on the plane. This shows
        that $\fund[\hyps]{H}\subseteq F$. Suppose, $\fund[\hyps]{H}\varsubsetneq F$. Then for some $\theta$ there is a point $p\in S_{\theta}$
        and $\gamma\in H$ with $\gamma p\in F\setminus\fund[\hyps]{H}$. Projecting back to $S$ gives a contradiction.
    \end{proof}
    \begin{remark}\label{rem:cocompact}
        For cocompact $\Gamma$ it is easy to see that $N(p,X)$ is uniformly bounded in $p$. On the other hand, for cofinite $\Gamma$
        it is not true in general, although it is still possible to see that $N(p,X)$ is well-defined (finite for fixed $X$ and $p$).
        For example, if $\Gamma=\pslzi$ then as $y(p)\tendsto\infty$, $N(p,X)$ becomes unbounded.
        This introduces complications for the convergence of the corresponding automorphic form and is the main
        reason for our restriction to cocompact groups. The problem lies in the fact that for non-compact $M$
        the totally geodesic
        surface can pass through the cusp. It should be possible to overcome this difficulty by restricting
        the group $\Gamma$ appropriately.
    \end{remark}

    \section{Spectral Analysis}\label{sec:spec}

    Let $\Gamma\subset\pslc$ be cocompact.
    Define a function $A(f)$ on $\hyps$ by
    \[A(f)(p) = \sum_{\gamma\in\cosets}f\left(\frac{1}{\cos^{2}v(\gamma p)}\right),\]
    where $\func{f}{[1,\infty)}{\reals}$ has a compact support and finitely many discontinuities.
    Then it is easy to see that $A(f)$ is automorphic, since $v(\gamma p)$ is constant
    on the cosets. Since $A$ is sufficiently smooth (in $C^{2}(M)$), it follows by~\cite[pg.~23]{huber1956} that $A$ converges
    pointwise to its spectral expansion.
    Let $\{u_{j}\}_{j\geq 0}$ be a complete orthonormal system of automorphic eigenfunctions of $-\Delta$ with corresponding eigenvalues $\lambda_{j}$.
    Since our problem differs from the standard lattice point counting problem, in that $A(f)$ does not define an automorphic kernel, we do not
    have the usual expansion in terms of the Selberg transform of $f$. The correct substitute for this is the spectral expansion of $A(f)$ in term
    of the $u_{j}$'s.
    Let $a(f,t_{j})$ be the coefficients of the spectral expansion of $A(f)$ on $\hypmodg$ given by
    \[a(f,t_{j}) = \langle A(f),u_{j}\rangle= \int_{\hypmodg}A(f)(p)\overline{u}_{j}(p)\d{\mu(p)}.\]
    Then the spectral expansion of $A(f)$ in terms of the $u_{j}$'s is
    \[A(f)(p) = \sum_{j} a(f,t_{j})u_{j}(p).\]
    We now compute the coefficients $a(f,t_{j})$ explicitly in the manner of \cite[Lemma~2.3]{huber1998} and \cite[Lemma~2.1]{chatzakos2015}.
    Following this, we identify the special function that appears in the spectral expansion and prove some simple estimates on it.
    For simplicity, consider $u_{j}$ instead of
    $\overline{u}_{j}$.
    \begin{lemma}\label{lem:specexp}
        We have
        \begin{equation*}
            a(f,t_{j}) = 2\period_{j} c(f,t_{j}),
        \end{equation*}
        where
        \[\period_{j}=\int_{H\backslash\hypp}u_{j}(x,u,0)\frac{du\,dx}{e^{u}},\]
        is a period-integral of $u_{j}$ over the fundamental domain of $H$ restricted to the plane $\hypp$.
        Also,
        \[c(f,t_{j})=\int_{0}^{\frac{\pi}{2}}f\left(\frac{1}{\cos^{2}v}\right)\frac{\xi_{\lambda_{j}}(v)}{\cos^{3}v}\d{v},\]
        where $\xi_{\lambda}$ is the solution of the ordinary differential equation
        \[\cos^{2}v\,\xi_{\lambda}''(v)+\sin v\cos v\,\xi_{\lambda}'(v)+\lambda\xi_{\lambda}(v)=0,\]
        with the initial conditions
        \[\xi_{\lambda}(0)=1,\quad\xi_{\lambda}'(0)=0.\]
    \end{lemma}
    In the following proofs we work with a fixed $\lambda$ and denote $\xi_{\lambda}$ by $\xi$.
    \begin{proof}
        Unfolding the spectral coefficients,
        \begin{align*}
            a(f,t_{j}) &= \int_{\hypmodg}\sum_{\gamma\in\cosets}f\left(\frac{1}{\cos^{2}v(\gamma p)}\right)u_{j}(p)\d{\mu(p)}\\
            &= \int_{H\backslash\hyps}f\left(\frac{1}{\cos^{2}v}\right)u_{j}(x,u,v)\frac{dx\,du\,dv}{e^{u}\cos^{3}v}.
        \end{align*}
        We can express this in terms of the period integral as
        \[a(f,t_{j})=\int_{-\frac{\pi}{2}}^{\frac{\pi}{2}} f\left(\frac{1}{\cos^{2}v}\right)\varphi_{j}(v)\frac{dv}{\cos^{3} v},\]
        where
        \[\varphi_{j}(v) = \int_{S_{v}} u_{j}(x,u,v)\frac{du\,dx}{e^{u}}.\]
        It is immediate that $\varphi_{j}$ is even.
        According to Lemma~\ref{lem:coords}, in our new coordinates the eigenvalue equation becomes:
        \[e^{2u}\cos^{2}v\pd[2][u_{j}]{x} + \cos^{2}v\left(\pd[2][u_{j}]{u}+\pd[2][u_{j}]{v}\right)-\cos^{2}v\pd[][u_{j}]{u}+\sin v\cos v\pd[][u_{j}]{v}+\lambda u_{j} = 0.\]
        Now, dividing by $e^{u}$ and integrating over $S_{v}$ we get
        \begin{multline}\label{eigeneq}
            \cos^{2}v\int_{S_{v}}e^{2u}\pd[2][u_{j}]{x}\frac{dx\,du}{e^{u}}+\cos^{2}v\int_{S_{v}}\pd[2][u_{j}]{u}\frac{du\,dx}{e^{u}}+\cos^{2}v\pd[2][\varphi_{j}]{v}\\
            -\cos^{2}v\int_{S_{v}}\pd{u}\int u_{j}\frac{dx\, du}{e^{u}}+\sin v\cos v\pd[][\varphi_{j}]{v}+\lambda\varphi_{j}(v)=0.
        \end{multline}
        Next, notice that the Laplacian on $S_{v}$ in the induced metric is exactly the restriction of $\Delta$ to $S_{v}$, that is,
        $\Delta\restriction_{S_{v}}=\cos^{2}v \Delta_{S_{v}}$, where
        \[\Delta_{S_{v}}=e^{2u}\pd[2]{x}+\pd[2]{u}-\pd{u}.\]
        Hence, for a fixed $v$,~\eqref{eigeneq} becomes
        \begin{equation}\label{eq:eigeneq2}
            \cos^{2}v\int_{S_{v}}\Delta_{S_{v}}u_{j}\frac{dx\,du}{e^{u}}+\cos^{2}v\pd[2][\varphi_{j}]{v}+\sin v\cos v\pd[][\varphi_{j}]{v}+\lambda\varphi_{j}(v)=0.
        \end{equation}
        Denote the integral in~\eqref{eq:eigeneq2} by $I_{v}$. Then, by Stokes' theorem we have that
        \[I_{v}=\int_{\partial S_{v}}\nabla u_{j}\cdot\vec{n}\d{\ell},\]
        where $d\ell$ is the line element on $S_{v}$ and $\vec{n}$ is the unit normal vector (on the plane) to $S_{v}$. We wish to show that $I_{v}=0$.
        To do this, recall some basic terminology for fundamental domains (in $\uph$) from \textcite{beardon1983}.
        Let $\fund{}$ be a fundamental domain of a cofinite or cocompact Fuchsian group $G$.
        Then $\fund{}$ is a convex hyperbolic polygon with finitely many sides.
        A \emph{side} of $\fund{}$ is a geodesic segment of the form $\overline{\fund{}}\cap g\overline{\fund{}}$ for any $g\in G$ with $g\neq I$.
        A \emph{vertex} of $\fund{}$ is a point of the form $\overline{\fund{}}\cap g\overline{\fund{}}\cap h\overline{\fund{}}$ for any $g\neq h\in G$ such that
        $g,h\neq I$. If $\partial\fund{}$ contains an elliptic fixed point of $g\in G$ of order 2, then we consider the fixed point
        as a vertex of $\fund{}$ and moreover $g$ identifies the adjacents sides with opposite orientation. In general, we can always find
        a side-pairing for $\fund{}$, that is, for $i=1$, $\ldots$, $k$, there exist triples $(\Lambda_{i},\Psi_{i},g_{i})$ such that
        $g_{i}\Lambda_{i}=\Psi_{i}$, and $g_{i}$ is the unique element in $G$ that does this, and that $\Lambda_{i}$ or $\Psi_{i}$ are not
        paired with any other sides of $\fund{}$. Finally, we can always choose $\fund{}$ so that if we consider $\partial\fund{}$ as a contour
        in $\hyps$, then the congruent sides occur with opposite orientation as segments of the contour~\cite[pp.~2--4]{hejhal1983}.
        So, let $\Set{(\Lambda_{i},\Psi_{i},g_{i})\given i=1,\ldots, k}$ be a side-pairing of $S$. Then it immediately follows that
        for any $S_{v}$ we get a corresponding side-pairing. Denote these by $\Set{(\Lambda_{i}^{v},\Psi_{i}^{v},g_{i})\given i=1,\ldots,k}$,
        where $\proj{(\Lambda_{i}^{v})}=\Lambda_{i}$ and $\proj{(\Psi_{i}^{v})}=\Psi_{i}$. It follows that $I_{v}=0$ as the integral over $\Lambda_{i}^{v}$
        is cancelled by the one over $\Psi_{i}^{v}$ since $\nabla u_{j}\cdot\vec{n}$ is invariant under $H$.
        We are left with
        \begin{equation}\label{eq:ode}
            \cos^{2}v\,\varphi''(v)+\sin v\cos v\, \varphi'(v)+\lambda\varphi(v)=0,
        \end{equation}
        where $\varphi'(0)=0$, as $\varphi$ is even.
        Define,
        \[\omega(v)=\varphi(v)+\varphi(-v),\]
        for $v\in(-\pi/2,\pi/2)$.
        Hence, adding~\eqref{eq:ode} evaluated at $-v$ to itself yields
        \begin{equation}\label{eq:fullode}
            \cos^{2}v\,\omega''(v) + \sin v\cos v\,\omega'(v) + \lambda\omega(v)=0,
        \end{equation}
        with $\omega(0) = 2\period_{j}$ and $\omega'(0) = 0$.
        Now, suppose that $\xi(v)$ is a solution to the second order homogenous linear ODE
        \begin{equation}\label{ode}
            \cos^{2}v\,\xi''(v)+\sin v\cos v\,\xi'(v)+\lambda\xi(v)=0,
        \end{equation}
        with initial conditions $\xi(0) = 1$ and $\xi'(0) = 0$.
        Then we can write the full solution $\omega$ of~\eqref{eq:fullode} as
        \[\omega(v)=2\period_{j}\xi(v).\]
        Therefore, the $a(f,t_{j})$'s can be written as
        \begin{align*}
            a(f,t_{j}) &=\left(\int_{0}^{\pi/2}+\int_{-\pi/2}^{0}\right)f\left(\frac{1}{\cos^{2}v}\right)\varphi(v)\frac{dv}{\cos^{3}v}\\
            &=2\period_{j}\int_{0}^{\pi/2}f\left(\frac{1}{\cos^{2}v}\right)\xi(v)\frac{dv}{\cos^{3}v}.
        \end{align*}
    \end{proof}
    We will also need some estimates on $\xi_{\lambda}$. Notice that the following lemma does not use the explicit
    form of $\xi_{\lambda}$ (which we will compute later). This computation is analogous to~\cite[\S4.2]{huber1998}.
    \begin{lemma}
        For all $v\in[0,\pi/2)$ we have
        \begin{align}
            \abs{\xi_{\lambda}(v)}&\leq 1,\label{eq:xibound1} \\
            \xi_{\lambda}(v)&\geq 1-\frac{2+\lambda}{2}\tan^{2}v.\label{eq:xibound2}
        \end{align}
    \end{lemma}
    \begin{proof}
        Multiplying~\eqref{ode} by $2\xi'(v)$, we can write
        \[\cos^{2}v \left(\xi'(v)^{2}\right)'+2\sin v\cos v (\xi'(v))^{2}+\lambda(\xi(v)^{2})'=0.\]
        Now, integrate over $[0,x]$ and use $\xi(0)=1$ and $\xi'(0)=0$
        \[\lambda(1-\xi(x)^{2})=\cos^{2}x \xi'(x)^{2}+2\int_{0}^{x}\sin 2v\, \xi'(v)^{2}\d{v}\geq 0,\]
        since $x\in[0,\pi/2)$ so that $\sin 2v$ is non-negative.
        This proves the first part.
        Now we can apply~\eqref{eq:xibound1} to get
        \begin{align*}
            \xi''(v) &\geq -\tan v \xi'(v)-\lambda\sec^{2}v.
        \end{align*}
        Integrating twice over $[0,x]$ and $[0,v]$ yields
        \[\xi(v)-1\geq(2+\lambda)\log(\cos v)\geq \frac{-1}{2}(2+\lambda)\tan^{2}v.\]
        Thus,
        \[\xi(v)\geq 1-\frac{2+\lambda}{2}\tan^{2}v.\]
    \end{proof}
    With this, we have the following Hecke type bound for the mean square of the period integrals.
    \begin{lemma}\label{lem:period}
        Let $\period_{j}$ be the period integral over $H\backslash\hypp$ of the automorphic form $u_{j}\in L^{2}(M)$. Then, for $T>1$
        \[ \sum_{\abs{t_{j}}\leq T}\abs{\period_{j}}^{2}\ll T.\]
    \end{lemma}
    This is a surprising result in the sense that the order of growth is better than what we expect from the local Weyl law.
    We suspect that the mean square should be bounded in all dimensions, cf.~\textcite[Theorem~1]{tsuzuki2009}.
    The proof is analogous to~\cite[\S2.6]{huber1998}.
    \begin{proof}
        Let $K=\sup_{p\in\hyps}N(p,X)$. This is well-defined as $N(p,X)$ is uniformly bounded.
        Then
        \[\int_{\hypmodg}(A(f)(p))^{2}\d{\mu(p)}\leq K\int_{\hypmodg}A(f)(p)\d{\mu(p)}.\]
        Also, define
        \[\tan\theta = \sqrt{\frac{2}{y+2}}.\]
        Then
        \[a(f,t_{j})=2\period_{j}\int_{0}^{\theta}\frac{\xi_{\lambda_{j}}(v)}{\cos^{3} v}\d{v}\]
        with $\xi_{0}(v)=1$. In particular, the coefficient for the zero eigenvalues (so $t_{j}=i$) gives
        \begin{align*}
            a(f, t_{0}) &= 2 \period_{0}\int_{0}^{\theta}\cos^{-3}v\d{v}\\
            &= 2 \period_{0}\left(\frac{1}{\sqrt{2}}\frac{\sqrt{4+y}}{y+2}+\frac{1}{2}\log\abs*{\frac{\sqrt{2}+\sqrt{4+y}}{\sqrt{y+2}}}\right)\\
            &= 2 \period_{0}g(y),
        \end{align*}
        say. On the other hand,
        \[\frac{1}{2}a(f,t_{0})=\period_{0}g(y)=u_{0}\int_{\hypmodg}A(f)(p)\d{\mu(p)}.\]
        It follows that
        \begin{equation}\label{eq:pars1}
            \int_{\hypmodg} (A(f)(p))^{2}\d{\mu(p)} \leq K\frac{\period_{0}g(y)}{u_{0}},
        \end{equation}
        where $u_{0}$ is the constant eigenfunction.
        By Parseval we have
        \[\int_{\hypmodg} (A(f)(p))^{2}\d{\mu(p)} = \sum_{j=0}^{\infty}\abs{a(f,t_{j})}^{2}\geq \sum_{\lambda_{j}\leq y}\abs{a(f,t_{j})}^{2},\]
        so that
        \begin{equation}\label{eq:pars2}
            \int_{\hypmodg} (A(f)(p))^{2}\d{\mu(p)} \geq \sum_{\lambda_{j}\leq y}\abs{\period_{j}}^{2}\left(\int_{0}^{\theta}\frac{\xi_{\lambda_{j}}(v)}{\cos^{3}v}\d{v}\right)^{2}.
        \end{equation}
        Since $\lambda_{j}\leq y$, from the bound~\eqref{eq:xibound2} it follows that:
        \[\xi_{\lambda_{j}}\geq 1-\frac{2+\lambda_{j}}{2}\tan^{2}v\geq 1-\frac{2+y}{2}\tan^{2}v.\]
        Hence,
        \begin{align*}
            \int_{0}^{\theta}\frac{\xi_{\lambda_{j}}(v)}{\cos^{3}v}\d{v} &\geq\int_{0}^{\theta}\left(1-\frac{2+y}{2}\tan^{2}v\right)\frac{dv}{\cos^{3}v}\\
            &=\frac{1}{8\tan^{2}\theta}\left(\tan\theta \sec\theta (2\sec^{2}\theta - 3)+(1+4\tan^{2}\theta)\log\left(\frac{1+\tan\frac{\theta}{2}}{1-\tan\frac{\theta}{2}}\right)\right)\\
            &=h(\theta).
        \end{align*}
        We are interested in the behaviour of $h$ as $y\tendsto\infty$, that is, $\theta\tendsto0^{+}$.
        After a tedious but elementary computation we find that
        \[\lim_{\theta\tendsto0^{+}}\frac{h(\theta)}{\tan\theta}=\frac{2}{3}.\]
        This means that $h(\theta)\geq(\frac{2}{3}-\epsilon)\tan\theta$ for small enough $\theta$ and for some $\epsilon>0$. In other words, we have proved that
        \begin{equation}\label{eq:pars3}
            \int_{0}^{\theta}\frac{\xi_{\lambda_{j}}(v)}{\cos^{3}v}\d{v}\geq cy^{-1/2},
        \end{equation}
        for some constant $c>0$, as $y\tendsto\infty$. Now, combining~\eqref{eq:pars1},~\eqref{eq:pars2} and~\eqref{eq:pars3} we get
        \[\sum_{\lambda_{j}\leq y}\abs{\period_{j}}^{2}\leq \frac{K\period_{0}}{c^{2}u_{0}}g(y)y\ll y^{1/2}.\]
        The result follows from observing that $\lambda_{j}=1+t_{j}^{2}$.
    \end{proof}
    As pointed out earlier, we can actually express $\xi$ in an elementary form. We suspect that this is always possible in
    odd dimensional hyperbolic space. In even dimensions the special functions are more complicated Legendre or hypergeometric functions.
    Let $r=\tan v$, then
    \begin{equation}\label{fourierc}
        a(f,t_{j})=2\period_{j}\int_{0}^{\infty}f(1+r^{2})\xi(\arctan r)\sqrt{1+r^{2}}\d{r}.
    \end{equation}
    Apply the transformation $\tan v = \sinh w$ in~\eqref{ode}. It becomes
    \[\xi''(w)+2\tanh w\, \xi'(w)+\lambda\xi(w)=0.\]
    Since $\lambda =s_{j}(2-s_{j})$, we have $1-\lambda=(s_{j}-1)^{2}$. It is then easy to see that the solution with our initial conditions is
    \begin{equation}\label{eq:xiw}
        \xi(w)=\frac{\cosh w(s_{j}-1)}{\cosh w},
    \end{equation}
    or in terms of $r$,
    \begin{equation}\label{eq:xir}
        \xi(r)=\frac{\cosh((s_{j}-1)\arcsinh r)}{\sqrt{1+r^{2}}}.
    \end{equation}
    Thus we can finally write the explicit form for $\xi$ in terms of $v$ as
    \begin{equation}\label{eq:xiv}
        \xi(v)=\frac{\cosh((s_{j}-1)\arcsinh\tan v)}{\sec v}.
    \end{equation}

    We will now show how to estimate the spectral coefficients $a(f,t_{j})$.
    With the explicit form~\eqref{eq:xir} for $\xi$, we can write
    \begin{equation}
        a(f,t_{j})=2\period_{j}\int_{0}^{\infty}f(1+r^{2})\cosh((s_{j}-1)\arcsinh r)\d{r}.
    \end{equation}
    We are thus led to consider the integral transform
    \[c(f,t)=\int_{0}^{\infty}f(1+r^{2})\cosh((s-1)\arcsinh r)\d{r},\]
    where $s=1+it$.
    Now, define
    \[f\left(\frac{1}{\cos^{2}v(p)}\right)=\begin{cases} 1,&\text{if $0\leq \abs{v}\leq \Theta$,}\\ 0,&\text{if $\Theta<\abs{v}<\frac{\pi}{2}$,}\end{cases}\]
    or equivalently
    \[f\left(\frac{1}{\cos^{2}v(p)}\right)=\begin{cases} 1,&\text{if $1\leq \sec v\leq X $,}\\ 0,&\text{if $X<\sec v$.}\end{cases}\]
    If we let $r=\tan v$, then we get that
    \[f\left(1+r^{2}\right)=\begin{cases} 1,&\text{if $0\leq r\leq U $,}\\ 0,&\text{if $U<r$,}\end{cases}\]
    where $U=\tan\Theta=\sqrt{X^{2}-1}$. Notice that
    \[U = X\sqrt{1-X^{-2}}=X\left(1+O(X^{-2})\right)=X+O(X^{-1}).\]
    In particular,
    \[A(f)(p)=\widetilde{N}(p,\Theta)=N(p,X).\]
    Now, letting $r=\sinh u$, we can rewrite $c(f,t)$ as
    \[c(f,t)=\int_{0}^{\infty}f(\cosh^{2}u)\cosh((s-1)u)\cosh u\d{u}.\]
    Notice that $2\cosh((s-1)u)\cosh u=\cosh su+\cosh(2-s)u$, so that
    \[c(f,t)=\frac{1}{4}\int_{\reals}f(\cosh^{2}u)\cosh su\d{u}+\frac{1}{4}\int_{\reals}f(\cosh^{2}u)\cosh (2-s)u\d{u},\]
    where
    \[f(\cosh^{2}u) =\begin{cases} 1, &\text{ if $\abs{u}\leq\arcsinh U$,}\\ 0, &\text{otherwise.}\end{cases}\]
    Since both of the integrals in $c(f,t)$ are of the same type, we define the integral transform $d(f,s)$ given by
    \begin{equation}\label{eq:dtransform}
        d(f,s)=\int_{\reals}f(\cosh^{2}u)\cosh su\d{u}.
    \end{equation}
    We list some simple properties of the $d(f,s)$-transform without proof.
    \begin{lemma}\label{lemma:dtransform}
        Suppose $f$ and $g$ are compactly supported even functions with finitely many discontinuities, let $\alpha\in\reals$, then
        \begin{align*}
            4c(f,t) &= d(f,s) + d(f,2-s),\\
            d(\alpha f, s)&=\alpha d(f,s),\\
            d(f\ast g,s)&=d(f,s)d(g,s),\\
            \intertext{where $\ast$ is the usual convolution. Also}
            d(\ind_{[-T,T]},s)&=\frac{2\sinh sT}{s},\\
            \intertext{where $\ind_{[-T,T]}(\cosh^{2}u)$ is the indicator function on $[-T,T]$, and}
            d(f,0)&=\int_{\reals}f(\cosh^{2}u)\d{u}.\\
        \end{align*}
    \end{lemma}
    Let $1>\delta>0$, and define $\chi(\cosh^{2}u)=(2\delta)^{-1}\ind_{[-\delta,\delta]}(\cosh^{2}u)$ to be a characteristic function with unit mass with respect to the
    $d(f,s)$-transform.
    Now define,
    \[\widetilde{f}^{+}(\cosh^{2}u)=\begin{cases}
            1, & \text{if $\abs{u}\leq \arcsinh U+2\delta$,}\\
            0, & \text{otherwise.}
        \end{cases}\]
    \[\widetilde{f}^{-}(\cosh^{2}u)=\begin{cases}
            1, & \text{if $\abs{u}\leq \arcsinh U-2\delta$,}\\
            0, & \text{otherwise.}
        \end{cases}\]
    Let $f^{+}=\widetilde{f}^{+}\ast\chi\ast\chi$ and $f^{-}=\widetilde{f}^{-}\ast\chi\ast\chi$.
    Then $f^{+}(x)=1$ for $\abs{x}\leq \arcsinh U$ and vanishes for $\abs{x}\geq \arcsinh U+4\delta$, and similarly
    $f^{-}(x)$ vanishes for $\abs{x}\geq \arcsinh U$. It follows that
    \[A(f^{-})(p) \leq N(p,X) \leq A(f^{+})(p).\]
    Hence, in order to estimate $N(p,X)$ we need bounds for $A(f^{+})$ and $A(f^{-})$, which in turn leads us to
    investigate $c(f^{\pm},t)$. The case for $f^{-}$ is analogous, so we restrict the treatment below to $f^{+}$.
    \begin{remark}
        Without any smoothing, the spectral expansion for $A(f)$ would of course not converge. In two dimensions
        it suffices to use a single convolution (linear decay). In our case we need at least two convolutions to
        ensure convergence.
        On the other hand, any more smoothing in this manner does not yield improvements for the pointwise bound
        nor for the application of the large sieve.
    \end{remark}
    \begin{prop}\label{prop:estimates}
        The integral transform $c(f,t)$ satisfies the following properties:
        \begin{enumerate}[(i)]
            \item\label{prop:estimates1} For $s=1+it$ we can write
                \begin{equation}\label{eq:oscillatory}
                    c(f^{+},t) = a(t, \delta)X^{1+it} + b(t,\delta)X^{1-it},
                \end{equation}
                where $a$ and $b$ satisfy
                \[a(t,\delta),\, b(t,\delta) \ll \min(\abs{t}^{-1},\abs{t}^{-3}\delta^{-2}).\]
            \item\label{prop:estimates2} For $s\in[1,2]$ we have
                \[c(f^{+},t) = \frac{2^{s-2}}{s}X^{s}+\frac{2^{-s}}{2-s}X^{2-s}+O(\delta X^{s}),\]
                where the case of $s=2$ is understood as
                \[c(f^{+},i) = \frac{X^{2}}{2}+O(\delta X^{2}).\]
        \end{enumerate}
    \end{prop}
    \begin{proof}
        We have
        \begin{align*}
            d(f^{+},s) &=\frac{8\sinh s(\arcsinh U+2\delta)\sinh^{2}s\delta}{(2\delta)^{2}s^{3}}\\
            &= \frac{8\sinh^{2}s\delta}{(2\delta)^{2}s^{3}}\sinh (s\log(U+\sqrt{U^{2}+1}) +2s\delta).
        \end{align*}
        By Taylor expansion $U+\sqrt{U^{2}+1}=2U+O(U^{-1})$, so that
        \[(U+\sqrt{U^{2}+1})^{s}=(2U)^{s}+O(sU^{s-2})=(2X)^{s}+O(sX^{s-2}).\]
        Now suppose $s\in[1,2)$, then we may assume that $\abs{s}\delta<1$. So,
        \begin{align*}
            \sinh s(\arcsinh U +2\delta) 
            &= \frac{1}{2}\left((2X)^{s}+O(X^{s-2})\right)(1+O(\delta))+O(\delta X^{s}+\delta X^{s-2})\\
            &= \frac{1}{2}(2X)^{s}+O(\delta X^{s}).
        \end{align*}
        It follows that
        \[d(f^{+},s) =\frac{4\sinh^{2}s\delta}{(2\delta)^{2}s^{3}}\left((2X)^{s}+O(\delta X^{s})\right).\]
        Since $\abs{s}\delta<1$, we also have that $(\sinh s\delta)/s\delta=1+O(\delta)$, and
        \[d(f^{+},s) = \frac{1}{s}(1+O(\delta))((2X)^{s}+O(\delta X^{s}))=\frac{2^{s}}{s}X^{s}+O(\delta X^{s}).\]
        So
        \begin{equation*}
            c(f^{+},t) =\frac{2^{s-2}}{s}X^{s}+\frac{2^{-s}}{2-s}X^{2-s}+O(\delta X^{s}).
        \end{equation*}
        Now, for the smallest eigenvalue, $s=2$, we get
        \[c(f^{+},i) = \frac{1}{4}\left( 2X^{2}+O(\delta X^{2}) + O(\log X)\right),\]
        as $d(\chi,0) = 1$.
        This proves~(\ref{prop:estimates2}) in the proposition.
        We now consider the case when $s$ is complex, that is, $s=1+it$. Assume $t>0$ and $X>1$, to get
        \begin{multline*}
            \sinh((1+it)(\arcsinh U + 2\delta)) = \frac{1}{2}\biggl((2X+O(X^{-1}))^{1+it}e^{2\delta(1+it)}
            -(2X+O(X^{-1}))^{-1-it}e^{-2\delta(1+it)}\biggr).
        \end{multline*}
        Thus we can write
        \begin{equation}\label{eq:sinhestimate}
            \sinh((1+it)(\arcsinh U + 2\delta))=X^{1+it}\upsilon(t,\delta),
        \end{equation}
        where $\upsilon(t,\delta)$ is bounded for $0<\delta<1$ and any $t$. Hence,
        \[d(f^{+},1+it)=X^{1+it}\upsilon(t,\delta)\frac{2}{s}\left(\frac{\sinh s\delta}{s\delta}\right)^{2}.\]
        Now, suppose $\abs{s}\delta<1$, then $\sinh(s\delta)/(s\delta)\ll 1$.
        So in this case
        \[d(f^{+},1+it)=X^{1+it}\abs{t}^{-1}.\]
        On the other hand, if $\abs{s}\delta\geq 1$, then $\sinh s\delta=O(1)$ so that
        \[d(f^{+},1+it)=X^{1+it}\abs{t}^{-3}\delta^{-2}.\]
        Working similarly with $d(f^{+},1-it)$ proves~(\ref{prop:estimates1}).
    \end{proof}
    Before we can prove the theorem, we need to know the local Weyl's law in our setting.
    \textcite{elstrodt1998} prove this for Eisenstein series in Chapter~6 Theorem~4.10.
    It is clear that their proof can be extended to include the cuspidal part. This yields the following lemma.
    \begin{lemma}\label{lem:cusp}For $T>1$, we have for all $p\in\hyps$ that
        \[\sum_{t_{j}\leq T}\abs{u_{j}(p)}^{2}\ll y(p)^{2}T+T^{3}.\]
    \end{lemma}
    We now have all the ingredients to prove our main theorem.
    \begin{proof}[Proof of Theorem~\ref{thm1}]
        First, write the spectral expansion of $A(f^{+})(p)$:
        \begin{align*}
            A(f^{+})(p) &= \sum_{j}2c(f^{+},t_{j})\period_{j}u_{j}(p)\\
            &= X^{2}\period_{0}u_{0}+\sum_{s_{j}\in[1,2)}2\period_{j}u_{j}(p)\left(\frac{2^{s_{j}-2}X^{s_{j}}}{s_{j}}+\frac{2^{-s_{j}}X^{2-s_{j}}}{2-s_{j}}+O(\delta X^{s_{j}})\right)\\
            &\phantom{=} +\sum_{t_{j}\in\reals}2c(f^{+},t_{j})\period_{j}u_{j}(p).
        \end{align*}
        Now the summation over $s_{j}\in[1,2)$ is finite, so
        \[A(f^{+})(p)=\sum_{s_{j}\in(1,2]}\frac{2^{s_{j}-1}X^{s_{j}}}{s_{j}}\period_{j}u_{j}(p)+G(f^{+},p)+O(X+\delta X^{2}),\]
        where
        \[G(f^{+},p) = \sum_{0\neq t_{j}}2c(f^{+},t_{j})\period_{j}u_{j}(p).\]
        Again, by the discreteness of the spectrum we can estimate the contribution of small~$t_{j}$'s
        \[G(f^{+},p) = \sum_{\abs{t_{j}}\geq 1}2c(f^{+},t_{j})\period_{j}u_{j}(p)+O(X).\]
        Now, since $c(f^{+},t)$ is even in $t$, we get by a dyadic decomposition
        \begin{align*}\label{eq:dyaestimate}
            \sum_{\abs{t_{j}}\geq 1}2c(f^{+},t_{j})\period_{j}u_{j}(p) &\ll \sum_{t_{j}\geq 1}c(f^{+},t_{j})\period_{j}u_{j}(p)\notag\\
            &= \sum_{n=0}^{\infty}\left(\sum_{2^{n}\leq t_{j}< 2^{n+1}}c(f^{+},t_{j})\period_{j}u_{j}(p)\right)\notag\\
            &\ll \sum_{n=0}^{\infty}\sup_{2^{n}\leq t_{j}<2^{n+1}}c(f^{+},t_{j})\left(\sum_{2^{n}\leq t_{j}<2^{n+1}}\period_{j}u_{j}(p)\right).
        \end{align*}
        By the Cauchy--Schwarz inequality and Lemmas~\ref{lem:period}~and~\ref{lem:cusp} we have
        \begin{align*}
            G(f^{+},p) &\ll \sum_{n=0}^{\infty}\sup_{2^{n}\leq t_{j}<2^{n+1}}c(f^{+},t_{j})\left(\sum_{t_{j}<2^{n+1}}\abs{\period_{j}}^{2}\right)^{1/2}\left(\sum_{t_{j}<2^{n+1}}\abs{u_{j}(p)}^{2}\right)^{1/2}+X\\
            &\ll \sum_{n=0}^{\infty}\sup_{2^{n}\leq t_{j}<2^{n+1}}c(f^{+},t_{j})2^{2n+2}+X.\\
            \intertext{We separate the sum over $n$ depending on whether $t_{j}\delta\leq 1$ or $t_{j}\delta\geq 1$,}
            G(f^{+},p)&\ll \sum_{n<\log_{2}\delta^{-1}}2^{2n+2}\sup_{2^{n}\leq t_{j}<2^{n+1}}c(f^{+},t_{j})+\sum_{n>\log_{2}\delta^{-1}}2^{2n+2}\sup_{2^{n}\leq t_{j}<2^{n+1}}c(f^{+},t_{j})+X.\\
            \intertext{Hence, by Proposition~\ref{prop:estimates},}
            G(f^{+},p)&\ll \sum_{n<\log_{2}\delta^{-1}}2^{2n+2}X2^{-n}+\sum_{n>\log_{2}\delta^{-1}}2^{2n+2}X\delta^{-2}2^{-3n}+X\ll X\delta^{-1}+X.
        \end{align*}
        Putting all this together we find that
        \begin{equation}\label{eq:finalestimate}
            A(f^{+})(p) = \sum_{s_{j}\in(1,2]}\frac{2^{s_{j}-1}X^{s_{j}}}{s_{j}}\period_{j}u_{j}(p)+O(X+\delta X^{2}+\delta^{-1} X).
        \end{equation}
        The optimal choice for $\delta$ comes from equating $\delta X^{2}=\delta^{-1} X$, which gives $\delta=X^{-1/2}$. The result follows from noting
        that $u_{0}=\vol(\hypmodg)^{-1/2}$ and $\period_{0}=\vol(S)\vol(\hypmodg)^{-1/2}$.
    \end{proof}

    \section{Applications of the Large Sieve}
    We will now apply Chamizo's large sieve inequalities to show that the mean square of the error term $\error[X][p]$ satifies the
    conjectured bound $O(X^{1+\epsilon})$ over a spatial average.
    In the radial aspect Chamizo proves large sieve inequalities with exponential weights for all moments in two dimensions.
    We extend his result to three dimensions for the second moment. We can only prove a mean square estimate of
    $O(X^{2+2/3})$ in the radial average. This translates to an improvement of $1/6$ compared to the pointwise
    bound we obtained in Section~\ref{sec:spec}.
    More specifically, our aim is to prove the following two theorems.
    \begin{thm}\label{thm:two}
        Let $X>2$ and $X_{1},\ldots,X_{R}\in[X,2X]$ such that $\abs{X_{k}-X_{l}}>\epsilon>0$ for all $k\neq l$.
        Suppose $R\epsilon\gg X$ and $R> X^{2/3}$, then
        \begin{equation}\label{eq:disclimit1}
            \frac{1}{R}\sum_{k=1}^{R}\abs{\error[X_{k}]}^{2}\ll X^{2+2/3}\log X.
        \end{equation}
        Taking the limit $R\rightarrow\infty$ gives
        \begin{equation}\label{eq:intlimit1}
            \frac{1}{X}\int_{X}^{2X}\abs{\error[x]}^{2}\d{x}\ll X^{2+2/3}\log X.
        \end{equation}
    \end{thm}
    For $p,q\in\hypmodg$, let
    \[\tilde{d}(p,q)=\inf_{\gamma\in\Gamma}d(p,\gamma q),\]
    be the induced distance on $\hypmodg$.
    \begin{thm}\label{thm:three}
        Let $X>2$ and $p_{1},\ldots,p_{R}\in\hypmodg$ with $\tilde{d}(p_{k},p_{l})>\epsilon>0$ for all $k\neq l$.
        Suppose $R\epsilon^{3}\gg 1$ and $R>X$, then
        \begin{equation}\label{eq:disclimit2}
            \frac{1}{R}\sum_{k=1}^{R}\abs{\error[X][p_{k}]}^{2}\ll X^{2}\log^{2}X.
        \end{equation}
        Taking the limit as $R\tendsto\infty$ gives
        \begin{equation}\label{eq:intlimit2}
            \int_{\hypmodg}\abs{\error[X][p]}^{2}\d{\mu(p)}\ll X^{2}\log^{2}X.
        \end{equation}
    \end{thm}
    We split the rest of this section into three parts: one for each of the averages after which we prove the generalisation of
    the radial large sieve inequality that is used for Theorem~\ref{thm:two}.

    \subsection{Radial Average}\label{sec:rad}
    We will prove the following proposition.
    \begin{prop}\label{prop:one}
        Let $X>2$ and $X_{1},\ldots,X_{R}\in[X,2X]$ such that $\abs{X_{k}-X_{l}}\geq \epsilon>0$ for all $k\neq l$. Then, we have
        \begin{equation}\label{eq:finalerror}
            \sum_{k=1}^{R}\abs{\error[X_{k}]}^{2} \ll  X^{3+2/3}\epsilon^{-1}\log X+X^{2+2/3}R+X^{3+1/3}\log X.
        \end{equation}
    \end{prop}
    Theorem~\ref{thm:two} follows immediately from the above proposition.
    \begin{proof}[Proof of Theorem~\ref{thm:two}]
        We take $\epsilon\asymp R^{-1}X$.
        Hence the bound~\eqref{eq:finalerror} becomes
        \[\sum_{k=1}^{R}\abs{\error[X_{k}]}^{2}\ll X^{2+2/3}R\log X+X^{2+2/3}R+X^{3+1/3}\log X.\]
        So if we choose $R>X^{2/3}$ then
        \[\frac{1}{R}\sum_{k=1}^{R}\abs{\error[X_{k}]}^{2}\ll X^{2+2/3}\log X.\]
        This proves~\eqref{eq:disclimit1}. For the integral limit~\eqref{eq:intlimit1} it suffices to consider a limiting partition of $[X,2X]$ with equally spaced points.
    \end{proof}
    The large sieve inequality for radial averaging is given by the following theorem.
    \begin{thm}\label{thm:chamizo2}
        Given $p\in\hypmodg$, suppose that $X>1$ and $T>1$. Let $x_{1},\ldots,x_{R}\in [X,2X]$. If $\abs{x_{k}-x_{l}}>\epsilon>0$ for all $k\neq l$, then
        \begin{equation}\label{eq:chamizo2}
            \sum_{k=1}^{R}\abs[\bigg]{\sum_{\abs{t_{j}}\leq T}a_{j}x^{it_{j}}_{k}u_{j}(p)}^{2}\ll (T^{3}+XT^{2}\epsilon^{-1})\norm{a}^{2}_{\ast},
        \end{equation}
        where
        \begin{equation}\label{eq:aeq}
            \norm{a}^{2}_{\ast}=\sum_{\abs{t_{j}}\leq T}\abs{a_{j}}^{2}.
        \end{equation}
    \end{thm}
    In two dimensions, \textcite[Theorem~2.2]{chamizo11996} proves a corresponding result to the above theorem.
    The proof in three dimensions is similar, but we write it down in Section~\ref{sec:lemmas}.
    Let $f$ be a compactly supported function with finitely many discontinuities on $[1,\infty)$, and denote
    \[\eerror[f] = A(f)(p) - \sum_{1\leq s_{j}\leq 2}c(f,t_{j})\period_{j}u_{j}(p).\]
    Then, recall we have shown that~(see~\eqref{eq:finalestimate})
    \begin{align*}
        \eerror[f^{+}] &= O(X\delta^{-1}+X),\\
        \eerror[f^{-}] &= O(X\delta^{-1}+X),
    \end{align*}
    and
    \begin{equation}\label{eq:newerror}
        \eerror[f^{-}] < \error+O(X^{2}\delta+X) < \eerror[f^{+}].
    \end{equation}
    We can now prove the proposition.
    \begin{proof}[Proof of Proposition~\ref{prop:one}]
        For simplicity, we combine the error terms in~\eqref{eq:newerror}. Suppose that
        $1>\delta\gg X^{-1}$, then
        \[\eerror[f^{-}] < \error+O(X^{2}\delta) < \eerror[f^{+}].\]
        Hence,
        \[\sum_{k=1}^{R}\abs{\error[X_{k}]}^{2}\ll \sum_{k=1}^{R}\abs{\eerror[f][X_{k}]}^{2}+RX^{4}\delta^{2},\]
        where $f$ is appropriately chosen as $f^{+}$ or $f^{-}$ depending on $k$. The main strategy is again
        to apply dyadic decomposition in the spectral expansion. We use the following notation
        for the truncated spectral expansion:
        \begin{equation}\label{eq:truncspec}
            \ssum = \sum_{T<\abs{t_{j}}\leq 2T}2c(f,t_{j})\period_{j}u_{j}(p).
        \end{equation}
        We consider three different ranges, which we choose so that the tails of the spectral expansion get
        absorbed into the error term. The correct ranges are given by
        \begin{align*}
            A_{1} &= \Set{t_{j}\given0<\abs{t_{j}}\leq 1},\\
            A_{2} &= \Set{t_{j}\given 1<\abs{t_{j}}\leq \delta^{-3}},\\
            A_{3} &= \Set{t_{j}\given \abs{t_{j}}>\delta^{-3}}.
        \end{align*}
        Also, define
        \[S_{i}=\sum_{t_{j}\in A_{i}}2c(f,t_{j})\period_{j}u_{j}(p).\]
        We can now write
        \[\eerror[f] = S_{1}+S_{2}+S_{3}.\]
        For the tail we have
        \begin{align*}
            \sum_{t_{j}\in A_{3}} 2c(f,t_{j})\period_{j}u_{j}(p) &\ll \sum_{\abs{t_{j}}\geq \delta^{-3}}c(f,t_{j})\period_{j}u_{j}(p)\\
            &\ll \sum_{t_{j}> \delta^{-3}}\min(\abs{t}^{-1},\abs{t}^{-3}\delta^{-2})X\period_{j}u_{j}(p).
        \end{align*}
        With a dyadic decomposition we get
        \begin{align*}
            \sum_{t_{j}\in A_{3}}2c(f,t_{j})\period_{j}u_{j}(p) &\ll X\delta^{-2}\sum_{n=0}^{\infty}\left(\sum_{2^{n}\delta^{-3}<t_{j}\leq 2^{n+1}\delta^{-3}}t^{-3}\period_{j}u_{j}(p)\right)\\
            &\ll X\delta^{-2}\sum_{n=0}^{\infty}\delta^{9}2^{-3n}\left(\sum_{2^{n}\delta^{-3}<t_{j}\leq 2^{n+1}\delta^{-3}}\period_{j}u_{j}(p)\right),\\
            \intertext{and then from the Cauchy--Schwarz inequality it follows that}
            &\ll X\delta^{7}\sum_{n=0}^{\infty}2^{-3n}\left(\sum_{t_{j}\leq
                    \delta^{-3}2^{n+1}}\abs{\period_{j}}^{2}\right)^{1/2}\left(\sum_{t_{j}\leq\delta^{-3}2^{n+1}}\abs{u_{j}(p)}^{2}\right)^{1/2}\\
            &\ll X\delta^{7}\sum_{n=0}^{\infty}2^{-3n}(2^{n/2}\delta^{-3/2})(2^{3n/2}\delta^{-9/2}) \ll X\delta,
        \end{align*}
        as required.
        Next, for the first interval we have
        \[ S_{1} = \sum_{t_{j}\in A_{1}}2c(f,t)\period_{j}u_{j}(p)\ll X\sum_{\abs{t_{j}}<1}\abs{t}^{-1}\period_{j}u_{j}(p)\ll X.\]
        Hence
        \[S_{1}+S_{3}=O(X^{2}\delta).\]
        Finally, we split the summation in $S_{2}$ into dyadic intervals by letting $T=2^{n}$ for $n=0, 1, \ldots, [\log_{2}\delta^{-3}]$.
        We then have
        \begin{equation}\label{eq:errorrad}
            \sum_{k=1}^{R}\abs{\eerror[f][X_{k}]}^{2}\ll\sum_{k=1}^{R}\abs*{\sum_{1\leq 2^{n}\leq \delta^{-3}}\ssum[2^{n}][X_{k}]}^{2}+O(RX^{4}\delta^{2}).
        \end{equation}
        Applying the Cauchy--Schwarz inequality to the sum over the dyadic intervals yields
        \[\abs[\bigg]{\sum_{1\leq 2^{n}<\delta^{-3}}\ssum[2^{n}][X_{k}]}^{2}\ll \log X \sum_{1\leq 2^{n}<\delta^{-3}}\abs{\ssum[2^{n}][X_{k}]}^{2}.\]
        Substituting this back into~\eqref{eq:errorrad} we get
        \begin{equation}\label{eq:errorest}
            \sum_{k=1}^{R}\abs{\eerror[f][X_{k}]}^{2}\ll\log X \sum_{1\leq 2^{n}\leq X^{-1}\delta^{-3}}\sum_{k=1}^{R}\abs{\ssum[2^{n}][X_{k}]}^{2}+O(RX^{4}\delta^{2}).
        \end{equation}
        Recall from Proposition~\ref{prop:estimates} that we can write
        \[c(f,t) = X(a(t,\delta)X^{it}+b(t,\delta)X^{-it}),\]
        where $a$ and $b$ satisfy
        \[a(t,\delta),\,b(t,\delta) = \min(\abs{t}^{-1},\abs{t}^{-3}\delta^{-2}).\]
        Keeping in mind our notation with $T=2^{n}$ and~\eqref{eq:truncspec}, we apply
        Theorem~\ref{thm:chamizo2} to get
        \begin{equation}\label{eq:sestimate2}
            \sum_{k=1}^{R}\abs{\ssum[T][X_{k}]}^{2}\ll (T^{3}+XT^{2}\epsilon^{-1})\norm{a}^{2}_{\ast},
        \end{equation}
        where
        \begin{align*}
            \norm{a}^{2}_{\ast} &\ll \sum_{T<\abs{t_{j}}\leq 2T}\abs{\min(\abs{t_{j}}^{-1},\abs{t_{j}}^{-3}\delta^{-2})X\period_{j}}^{2}\\
            &\ll X^{2}\min(T^{-2},T^{-6}\delta^{-4})\left(\sum_{T\leq \abs{t_{j}}<2T}\abs{\period_{j}}^{2}\right)\\
            &\ll X^{2}\min(T^{-1},T^{-5}\delta^{-4}).
        \end{align*}
        This simplifies~\eqref{eq:sestimate2} to
        \[\sum_{k=1}^{R}\abs{\ssum[T][X_{k}]}^{2}\ll (T^{3}+XT^{2}\epsilon^{-1})X^{2}\min(T^{-1},T^{-5}\delta^{-4}).\]
        Therefore~\eqref{eq:errorest} becomes
        \begin{align*}
            \sum_{k=1}^{R}\abs{\error[X_{k}]}^{2} 
            &\ll \log X\sum_{1\leq T<\delta^{-3}}(T^{3}+XT^{2}\epsilon^{-1})X^{2}\min(T^{-1},T^{-5}\delta^{-4})+RX^{4}\delta^{2}.
        \end{align*}
        We split the summation depending on whether $T<\delta^{-1}$ and get
        \begin{align*}
            \sum_{k=1}^{R}\abs{\error[X_{i}]}^{2} &\ll X^{2}\log X\left(\sum_{1\leq T\leq\delta^{-1}}T^{2}\right)+X^{3}\epsilon^{-1}\log X\left(\sum_{1\leq T\leq\delta^{-1}}T\right)\\
            &\phantom{\ll}+X^{2}\delta^{-4}\log X\left(\sum_{\delta^{-1}\leq T\leq \delta^{-3}}T^{-2}\right)\\
            &\phantom{\ll}+X^{3}\delta^{-4}\epsilon^{-1}\log X\left(\sum_{\delta^{-1}\leq T\leq \delta^{-3}}T^{-3}\right)+RX^{4}\delta^{2}.
        \end{align*}
        With trivial estimates we have
        \[\sum_{k=1}^{R}\abs{\error[X_{k}]}^{2} \ll X^{2}\delta^{-2}\log X+X^{3}\epsilon^{-1}\delta^{-1}\log X+RX^{4}\delta^{2}.\]
        The optimal choice for $\delta$ comes from $X^{4}\delta^{2}=X^{2}\delta^{-1}$, that is, $\delta=X^{-2/3}$, since $\epsilon R\asymp X$. This gives
        \begin{equation*}
            \sum_{k=1}^{R}\abs{\error[X_{k}]}^{2} \ll  X^{3+1/3}\log X+X^{3+2/3}\epsilon^{-1}\log X+RX^{2+2/3}.
        \end{equation*}
    \end{proof}

    \subsection{Spatial Average}\label{sec:spa}
    We now consider the spatial average. In this case the corresponding large sieve inequality was
    already proved by Chamizo in~\cite[Theorem~3.2]{chamizo11996} for $n$ dimensions for any cocompact group.
    It would not be difficult to extend it to cofinite groups in three dimensions.
    We state it as Theorem~\ref{thm:chamizo1} simplified to our setting. With similar strategy as in Section~\ref{sec:rad}, we prove the following
    proposition which readily yields Theorem~\ref{thm:three}.
    \begin{prop}\label{prop:two}
        Suppose $X>1$ and let $p_{1},p_{2},\ldots,p_{R}\in\hypmodg$ with $\tilde{d}(p_{k},p_{l})>\epsilon$ for
        some $\epsilon>0$. Then we have
        \[\sum_{k=1}^{R}\abs{\error[X][p_{k}]}^{2} \ll X^{4}R^{-1}\log^{2}X+X^{2}\epsilon^{-3}.\]
    \end{prop}
    As before, Theorem~\ref{thm:three} follows easily.
    \begin{proof}[Proof of Theorem~\ref{thm:three}]
        We pick $\epsilon^{-3}\ll R$ and $R>X$. Then
        \[\frac{1}{R}\sum_{k=1}^{R}\abs{\error[X][p_{k}]}^{2} \ll X^{4}R^{-2}\log^{2}X+X^{2}\ll X^{2}\log^{2}X.\]
        For the integral limit we take hyperbolic balls of radius $\epsilon/2$ uniformly spaced in $M$.
        For small radii the
        volume of such a ball is $(4/3)\pi (\epsilon/2)^{3}$,~\cite[pg.~10]{elstrodt1998}. This is compatible
        with our assumption that $R\epsilon^{3}$ is bounded from below since $M$ is of finite volume.
    \end{proof}

    For the proof of Proposition~\ref{prop:two} we use the large sieve inequality in the following form.
    \begin{thm}[Theorem~3.2~in~\cite{chamizo11996}]\label{thm:chamizo1}
        Given $T>1$, $p_{1},\ldots,p_{R}\in M=\hypmodg$, if\\ $\tilde{d}(p_{k},p_{l})>\epsilon>0$ for all $k\neq l$, then
        \[\sum_{k=1}^{R}\abs*{\sum_{\abs{t_{j}}\leq T} a_{j}u_{j}(p_{k})}^{2}\ll (T^{3}+\epsilon^{-3})\norm{a}_{\ast}^{2},\]
        where $\norm{a}_{\ast}$ is as in~\eqref{eq:aeq}.
    \end{thm}
    We can then prove Proposition~\ref{prop:two}.
    \begin{proof}[Proof of Proposition~\ref{prop:two}]
        By a direct application of the Cauchy--Schwarz inequality we have
        \[\left(\sum_{k=1}^{n}a_{k}\right)^{2} \ll \sum_{k=1}^{n}k^{2}a_{k}^{2}.\]
        Thus,
        \begin{equation}\label{eq:logbound}
            \abs*{\sum_{1\leq T\leq\delta^{-3}}\ssum[T][X][p_{k}]}^{2} \ll\sum_{1\leq T\leq\delta^{-3}}\abs{\log T}^{2}\abs{\ssum[T][X][p_{k}]}^{2}.
        \end{equation}
        Repeating this with the identity
        \[\left(\sum_{k=1}^{n}a_{k}\right)^{2} \ll \sum_{k=1}^{n}(n+1-k)^{2}a_{k}^{2},\]
        and combining with~\eqref{eq:logbound} yields
        \[\abs*{\sum_{1\leq T\leq\delta^{-3}}\ssum[T][X][p_{k}]}^{2} \ll\sum_{1\leq T\leq\delta^{-3}}\abs{c_{T}}^{2}\abs{\ssum[T][X][p_{k}]}^{2},\]
        where $c_{T}=\min\left(\abs{\log T\delta^{3}+1},\log T\right)$.
        Repeating the analysis from the proof of Proposition~\ref{prop:one} and applying Theorem~\ref{thm:chamizo1} gives
        \begin{align*}
            \sum_{k=1}^{R}\abs{\error[X][p_{k}]}^{2} &\ll \sum_{1\leq T\leq \delta^{-3}}\abs{c_{T}}^{2}(T^{3}+\epsilon^{-3})X^{2}T^{-1}\min(1,T^{-4}\delta^{-4})+RX^{4}\delta^{2}\\
            &\ll X^{2}\delta^{-2}\log^{2}X + X^{2}\epsilon^{-3} + RX^{4}\delta^{2},
        \end{align*}
        as before. We optimise by setting $X^{2}\delta^{-2}=RX^{4}\delta^{2}$, which gives $\delta = R^{-1/4}X^{-1/2}$. This yields
        \[\sum_{k=1}^{R}\abs{\error[X][p_{k}]}^{2} \ll X^{3}R^{1/2}\log^{2}X+X^{2}\epsilon^{-3}.\]
    \end{proof}

    \subsection{A Large Sieve Inequality}\label{sec:lemmas}
    We need two technical lemmas to prove Theorem~\ref{thm:chamizo2}. The first one is Lemma~3.2~in~\cite{chamizo11996}.
    \begin{lemma}\label{chamizolemma}
        Let $b=(b_{1},\ldots,b_{R})\in\complex^{R}$ be a unit vector and let $A=(a_{ij})$ be an $R\times R$ matrix over $\complex$ with
        $\abs{a_{ij}}=\abs{a_{ji}}$. Then
        \[\abs{b\cdot Ab}=\abs{\sum_{i,j=1}^{R}b_{i}\overline{b}_{j}a_{ij}}\leq \max_{i}\sum_{j=1}^{R}\abs{a_{ij}}.\]
    \end{lemma}
    We also need to compute the inverse Selberg transform for the Gaussian at different frequencies.
    \begin{lemma}\label{lem:selberg}
        Let $h(1+t^{2})=e^{-t^{2}/(2T)^{2}}\cos(rt)$. The inverse Selberg transform $k$ of $h$ satisfies for all $x>0$
        \begin{align*}
            k(\cosh x) &\ll T^{3}\frac{(x+r)e^{-T^{2}(x+r)^{2}}+(x-r)e^{-T^{2}(x-r)^{2}}}{\sinh x},
            \shortintertext{and}
            k(1) &\ll \min(T^{3},r^{-3}).
        \end{align*}
    \end{lemma}
    \begin{proof}[Proof of Lemma~\ref{lem:selberg}]
        According to~\cite[\S3~Lemma~5.5]{elstrodt1998}, the inverse Selberg transform $k$ of $h$ for $x\geq1$
        is given by
        \[-2\pi k(x)=\frac{d}{dx}\frac{1}{2\pi}\int_{-\infty}^{\infty}h(1+t^{2})e^{-it\arccosh x}\d{t}.\]
        Hence, by a direct computation
        \begin{equation*}
            -2\pi k(\cosh x) = \frac{1}{2\pi i}\frac{1}{\sinh x}\int_{-\infty}^{\infty}e^{-t^{2}/(2T)^{2}}\cos(rt)te^{-itx}\d{t},
        \end{equation*}
        which is just the Fourier transform of a Gaussian times $t\cos(rt)$. Hence, by standard results
        \cite[17.22~(2)]{gradshteyn2007} and~\cite[17.23~(13)]{gradshteyn2007}
        we have
        \[k(\cosh x) = \frac{2\sqrt{\pi}}{4\pi^{2}}T^{3}g(x),\]
        where
        \[g(x)=\frac{(x+r)e^{-T^{2}(x+r)^{2}}+(x-r)e^{-T^{2}(x-r)^{2}}}{\sinh x}.\]
        Taking the limit gives
        \begin{align*}
            \lim_{x\tendsto 0} g(x) 
            &=2e^{-T^{2}r^{2}}(1-2T^{2}r^{2}).
        \end{align*}
        Let $u(x)=2e^{-x^{2}}(1-2x^{2})$. Hence,
        \[k(1) = \frac{1}{2\pi^{3/2}}T^{3}u(rT).\]
        Since $u(x)$ is bounded, we get trivially that $k(1)\ll T^{3}$. On the other hand,
        \[T^{3}u(rT)=\frac{1}{r^{3}}(Tr)^{3}u(rT)\ll r^{-3}.\]
        It follows that $k(1)\ll\min(T^{3},r^{-3})$.
    \end{proof}
    \begin{proof}[Proof of Theorem~\ref{thm:chamizo2}]
        Let $S$ be the left-hand side of~\eqref{eq:chamizo2}.
        Since $\complex^{R}$ is self-dual, it follows from Riesz representation theorem that
        there exists a unit vector $\vec{b}=(b_{1},\ldots,b_{R})$ in $\complex^{R}$ such that
        \[ S = \left(\sum_{k=1}^{R}b_{k}\left(\sum_{\abs{t_{j}}\leq T}a_{j}x^{it_{j}}_{k}u_{j}(p)\right)\right)^{2}.\]
        Then, by the Cauchy--Schwarz inequality
        \[ S\leq \norm{a}^{2}_{\ast}\widetilde{S},\]
        where
        \[\widetilde{S} = \sum_{\abs{t_{j}}\leq T}\abs*{\sum_{k=1}^{R}b_{k}x^{it_{j}}_{k}u_{j}(p)}^{2}.\]
        In order to understand the sum $\widetilde{S}$, we smooth it out by a Gaussian centered around zero.
        This allows us to apply Lemma~\ref{lem:selberg}, which shows that the
        Selberg transform for a Gaussian is easy to compute. Thus,
        \[\widetilde{S} \ll \sum_{j}e^{-t_{j}^{2}/(4T^{2})}\abs*{\sum_{k=1}^{R}b_{k}x^{it_{j}}_{k}u_{j}(p)}^{2}.\]
        After we open up the squares and interchange the order of summation, we apply Lemma~\ref{chamizolemma} to get
        \[S\ll \norm{a}^{2}_{\ast}\max_{k}\sum_{l=1}^{R}\abs{S_{kl}},\]
        where
        \[S_{kl} = \sum_{j}e^{-t_{j}^{2}/(4T^{2})}\cos(r_{kl} t_{j})\abs{u_{j}(p)}^{2},\]
        and
        \[r_{kl} = \abs[\bigg]{\log\frac{x_{k}}{x_{l}}}.\]
        We can identify $S_{kl}$ as the diagonal contribution in the spectral expansion of an automorphic kernel with $h(1+t^{2})=e^{-t^{2}/(4T^{2})}\cos(r_{kl}t)$.
        It follows from Lemma~\ref{lem:selberg} that
        \begin{equation}\label{eq:specexp}
            S_{kl}\ll\min(T^{3},r_{kl}^{-3})+\sum_{\gamma\neq\id}T^{3}e^{-T^{2}(d(\gamma p, p)-r_{kl})^{2}}.
        \end{equation}
        The standard hyperbolic lattice point problem (e.g.~\cite[\S2~Lemma~6.1]{elstrodt1998}) gives
        \[\#\Set{\gamma\in\Gamma\given\delta(p,\gamma q)\leq x}\ll x^{2},\]
        where the implied constant depends on $\Gamma$ and $p$. We can rewrite this as
        \[\log(1+\#\Set{\gamma\in\Gamma\given r<d(p,\gamma q)\leq r+1})\ll r^{2}+1.\]
        This shows that the series in~\eqref{eq:specexp} converges as $T\tendsto\infty$, so that
        \[S_{kl}\ll\min(T^{3},r_{kl}^{-3}).\]
        Hence, by the mean value theorem
        \[\sum_{l=1}^{R}\abs{S_{kl}}\ll\sum_{l=1}^{R}\min(T^{3},X^{3}\abs{x_{k}-x_{l}}^{-3}).\]
        The case $l=k$ yields $T^{3}$. So suppose $l\neq k$, then
        separate the $x_{l}$ for which $T\leq X\abs{x_{k}-x_{l}}^{-1}$. By the spacing condition, there
        are at most $2XT^{-1}\epsilon^{-1}$ such points. Hence,
        \begin{align}
            \sum_{l=1}^{R}\abs{S_{kl}}&\ll T^{3}XT^{-1}\epsilon^{-1}+\int_{1}^{\infty}\frac{X^{3}}{\abs{XT^{-1}+\epsilon u}^{3}}\d{u} + T^{3}\label{eq:improvement}\\
            &\ll T^{2}X\epsilon^{-1}+T^{3}.\notag
        \end{align}
    \end{proof}

    \printbibliography
\end{document}